\documentclass[letterpaper, conference]{ieeeconf}
\usepackage[utf8]{inputenc}

\IEEEoverridecommandlockouts                              
\usepackage[utf8]{inputenc}

\usepackage{amsmath,amssymb}

\usepackage{amsthm}
\usepackage{dutchcal}
\usepackage{color}
\usepackage{todonotes}
\usepackage{bm}
\usepackage{soul}
\usepackage[ruled,vlined]{algorithm2e}
\usepackage{pgfgantt}
\usepackage{graphicx}
\usepackage{xcolor}
\usepackage{datetime}
\usepackage{standalone}
\newtheorem{Theorem}{\bf Theorem}

\newtheorem{Assumption}[Theorem]{\bf Assumption}
\newtheorem{Algorithm}[Theorem]{\bf Algorithm}

\newdateformat{monthdayyeardate}{\monthname[\THEMONTH]~\THEDAY, \THEYEAR}
\newdateformat{monthdayyeardatee}{\THEDAY.\THEMONTH.\THEYEAR}

\author{Francesco Micheli, John Lygeros
\thanks{Research supported by the European Research Council under the H2020 Advanced Grant no. 787845 (OCAL).}
\thanks{The authors are with the Automatic Control Laboratory, Swiss Federal Institute of Technology (ETH) Z\"{u}rich, 8092 Z\"{u}rich, Switzerland. Emails: \texttt{\{frmicheli, jlygeros\}@ethz.ch}.}%
}
\title{Scenario-based Stochastic MPC for systems with uncertain dynamics}

\begin{document}
\maketitle
\begin{abstract}
Model Predictive Control is an extremely effective control method for systems with input and state constraints. Model Predictive Control performance heavily depends on the accuracy of the open-loop prediction. For systems with uncertainty this in turn depends on the information that is available about the properties of the model and disturbance uncertainties. Here we are interested in situations where such information is only available through realizations of the system trajectories. We propose a general scenario-based optimization framework for stochastic control of a linear system affected by additive disturbance, when the dynamics are only approximately known. The main contribution is in the derivation of an upper bound on the number of scenarios required to provide probabilistic guarantees on the quality of the solution to the deterministic scenario-based finite horizon optimal control problem. We provide a theoretical analysis of the sample complexity of the proposed method and demonstrate its performance on a simple simulation example. Since the proposed approach leverages sampling, it does not rely on the explicit knowledge of the model or disturbance distributions, making it applicable in a wide variety of contexts.
\end{abstract}
\section{Introduction}
Model Predictive Control (MPC) has been effectively used to control complex dynamical systems in presence of input and state constraints, however, its performance is strongly affected by the accuracy of the prediction model. This aspect is critical when the system is too complex or too expensive to be accurately modelled and identification or learning algorithms are employed to obtain approximate plant dynamics.
In the last decades, growing attention has been devoted to stochastic and robust MPC methods that can guarantee stability and constraint satisfaction in the presence of model uncertainties and exogenous disturbances.

Robust MPC (RMPC) classically considers min-max formulations that optimize the control actions with respect to the worst-case model uncertainty or disturbance realizations \cite{wan2003efficient,kouvaritakis2000efficient,mayne2000constrained}. These approaches rely on the assumption of bounded uncertainties and on the convexity of the optimization problem with respect to both control variables and uncertainties. The constraints need to be satisfied for \textit{all} possible values of the uncertainties, an approach that can be conservative if additional information about the uncertainty and its distribution is available (for example, through samples). Furthermore, the computational burden of RMPC methods does not scale favorably with the model and disturbance complexities, sometimes requiring ad hoc simplifications of the constraints.

When an appropriate probabilistic description of the model uncertainties is available, Stochastic MPC (SMPC) \cite{cannon2010stochastic,cinquemani2009convex} can be used to enforce state and input limitations in the form of chance constraints. Admitting a small level of constraint violation allows one to reduce the conservatism with respect to the RMPC solution, especially in the presence of rare events. Similarly to RMPC, computational tractability typically requires approximations or assumptions on the probability distributions.

An alternative approach is provided by the class of randomized uncertainty methods \cite{alamo2010sample,vidyasagar2001randomized,calafiore2006probabilistic,shapiro2005complexity,calafiore2011research}. These approaches rely on sampling to obtain a deterministic reformulation of the chance constraint problem, without assumptions on the probability distribution of the uncertainties. A popular method is the so called \textit{scenario approach} that, in its basic formulation \cite{calafiore2006scenario,campi2008exact,calafiore2010random}, relies on the convexity of the optimization problem with respect to the decision variables to provide tight \textit{a-priori} bounds on the number of samples required to guarantee the original chance constraint satisfaction. The idea of applying the scenario approach in a MPC framework has been explored in, among others, ~\cite{prandini2012randomized} for known linear time invariant deterministic models affected by additive disturbance, \cite{calafiore2012robust} for parametric linear time invariant stochastic dynamics with general dependencies of the system on the parameter realization and \cite{calafiore2013stochastic} for linear parameter varying dynamics. In~\cite{margellos2014road} a robust optimization problem is proposed for which the uncertainty bounds are computed via the scenario approach.

Building on these ideas, we propose a general, but tractable, scenario-based optimization framework for chance constrained stochastic MPC with linear uncertain dynamics. Unlike earlier approaches, the proposed work robustifies, in a probabilistic sense, the stochastic MPC against the model uncertainty related to the model dynamics. For uncertain linear time invariant systems, we provide a probabilistic upper bound on the number of samples required to achieved a prescribed probability of violation, leading to finite sample randomized algorithms to solve the stochastic control problem. Since the proposed approach leverages sampling, it does not rely on the explicit knowledge of the model and disturbance distributions and it does not require any specific assumption on the uncertainties, making it compatible with a wide variety of probabilistic identification algorithms.

The rest of the paper is organized as follows. The problem formulation is described in Section~\ref{Sec:Problem_Formulation}. Section~\ref{Sec:Background_Scenario} reviews earlier results in scenario MPC. In Section~\ref{Sec:Uncertain_dynamics_Scenario_Stochastic_MPC} we derive an upper bound on the number of scenarios needed to achieve a prescribed violation probability level for the case of scenario MPC with uncertain dynamics. Section~\ref{Sec:Numerics} provides a numerical example. Section~\ref{Sec:Conclusion} concludes the paper. 
\section{Problem formulation}\label{Sec:Problem_Formulation}
Consider the discrete-time linear time invariant (LTI) system subject to additive disturbances
\begin{equation}\label{eq:Dynamics}
	x_{k+1}=\bar{A} x_{k} + \bar{B} u_{k} +\eta_k,
\end{equation}
with state $x_k \in \mathbb{R}^{n}$, control input $u_k \in \mathbb{R}^{m}$ and disturbance $\eta_k \in \mathbb{R}^{n}$ distributed according to a (possibly unknown) probability measure $\mathbb{P}_{\eta}$ over the (possibly unknown and unbounded) support set $\mathcal{H} \subseteq \mathbb{R}^{n}$.
We assume that the state is measurable, but that we do not have access to the true system dynamics. This is relevant in many engineering contexts where only an approximate model is available for the control task.
Here we summarize the knowledge we have about the system dynamics by considering an arbitrary parametrization of the system matrices $A = A(\vartheta)$ and $B = B(\vartheta)$ and by assuming a parameters distribution $\vartheta \sim \mathbb{P}_{\vartheta}$ defined over the (possibly unknown and unbounded) support set ${\Theta}$. We further assume that the true system dynamics are within this parametrized class, i.e., there exists $\bar{\vartheta} \in \Theta$ such that $ \bar{A}=A(\bar{\vartheta})$ and $\bar{B}=B(\bar{\vartheta})$. As it will be evident in the next sections, our approach does not require explicit knowledge of the parameter distribution $\mathbb{P}_{\vartheta}$ as long as i.i.d. samples from the distribution are available. An example where the parameter distribution is not explicitly known is provided in Section~\ref{Sec:Numerics}.

When only an approximate model of the dynamics is available, two sources of uncertainty can be distinguished:
\begin{itemize}
	\item Epistemic uncertainty on the system dynamics (i.e. the value of $\vartheta$ for the true system matrices are unknown);
	\item Stochastic uncertainty due to the presence of the disturbance $\eta_k$. 
\end{itemize}
The epistemic uncertainty is often treated in the literature by selecting an estimate using regression, then using this estimate in a problem formulation that addresses the stochastic uncertainty. We argue, however, this can be detrimental for performance whenever there is a significant mismatch between the model corresponding to the parameter estimate and the true system dynamics.
We aim, instead, for a problem formulation that can provide state and input constraint satisfaction with respect to both the epistemic and the stochastic uncertainty.

Our objective is to design a receding horizon predictive controller that minimizes a given cost function while satisfying probabilistic state and input constraints over a finite horizon $T$. To separately account for model uncertainty and for the effect of disturbances, at each time $\tau$ we would like to solve the finite horizon optimal control (FHOC) problem:
\begin{equation}\label{eq:MPC_problem_1}
	\begin{aligned}
		\min_{\boldsymbol{x}_+,\boldsymbol{u},\boldsymbol{\mu}} \quad & \mathbb{E}_{\vartheta,\eta}\left[ J\left( \boldsymbol{x}_{+}, \boldsymbol{u} \right) \right]\\
		\text{s.t.}\quad& x_{\tau+k+1}=A x_{\tau+k} + B u_{\tau+k} +\eta_{\tau+k}\\
		\quad & \boldsymbol{u}=\boldsymbol{\mu}\left( \boldsymbol{x}_{+} \right)\qquad\qquad\qquad k=0,\dots,T\!-\!1\\
		\quad & x_{\tau}=\bar{x}_{\tau}\\
		\quad & \mathbb{P}_{\vartheta}\left(\mathbb{P}_{\eta|\vartheta}\left(f\left( \boldsymbol{x}_+,\boldsymbol{u}\right)\leq 0 \right)\geq 1-\varepsilon_1\right)\geq 1-\varepsilon_2\\
	\end{aligned}
\end{equation}
where $\bar{x}_{\tau}$ is the initial condition at time $\tau$, 
\begin{equation*}
	\boldsymbol{x}_{+}=\begin{bmatrix}
		x_{\tau+1} \\
		x_{\tau+2} \\
		\vdots \\
		x_{\tau\!+\!T}
	\end{bmatrix}, \ \boldsymbol{u}=\begin{bmatrix}
		u_{\tau} \\
		u_{\tau+1} \\
		\vdots \\
		u_{\tau\!+\!T\!-\!1}
	\end{bmatrix}, \ \boldsymbol{\eta}=\begin{bmatrix}
		\eta_{\tau} \\
		\eta_{\tau+1} \\
		\vdots \\
		\eta_{\tau\!+\!T\!-\!1}
	\end{bmatrix},
\end{equation*}
$J\left( \boldsymbol{x}_{+}, \boldsymbol{u} \right)$ is a cost function and $\mu(\cdot):\mathbb{R}^{nT}\to\mathbb{R}^{mT}$ takes values in a class of causal policies ($\mu_{\tau+k}(\boldsymbol{x}_+)$ independent of $x_{\tau+k},\dots,x_{\tau+T}$).
The probabilistic constraint function $f\left(\boldsymbol{x}_+,\boldsymbol{u}\right):\mathbb{R}^{(n+m)T}\to\mathbb{R}$ can encode constraints that need to be satisfied with probabilities of violation $\varepsilon_1 \in [0,1]$ and $\varepsilon_2 \in [0,1]$. These can be either application specific requirements or design parameters that trade off performance for constraint satisfaction. Unlike robust constraint satisfaction, chance constraints can allow some violations to obtain a larger feasible set and a lower cost. This is a reasonable assumption in many applications where, by specification, limited constraint violations are allowed.

Solving~\eqref{eq:MPC_problem_1} over arbitrary control policies is, in general, an intractable problem. To maintain tractability, one can restrict attention to open-loop policies $u_{\tau},\dots,u_{\tau+T-1}$ that are trivially causal; this, however, can be too conservative, as it does not allow the input to exploit the information about the future disturbance that will be available when the decision is executed. A popular alternative is to consider causal state-affine policies that can be reformulated as a disturbance-affine policy~\cite{goulart2006optimization} to maintain computational tractability. Given the assumption of fully measurable state, the disturbance $\eta_k$ at each time step can be reconstructed according to
\begin{equation}\label{eq:dynamics_flip}
\eta_k = x_{k+1} - \left( A x_{k} + B u_{k} \right).
\end{equation}
The resulting disturbance-affine policy can be written as
\begin{equation}\label{eq:AffinePolicy}
    \left\{\begin{aligned}
        u_{\tau}&= \gamma_{\tau}\\
        u_{\tau + k}&= \boldsymbol{\gamma}_{\tau + k}+\sum_{j=0}^{k-1} \lambda_{k, j}  \eta_{\tau+j}\ ,\quad\ k=1,\dots,T\!-\!1\\
    \end{aligned}\right.
\end{equation}
where $\boldsymbol{\gamma}_{k} \in \mathbb{R}^{m}$ and $\lambda_{k} \in \mathbb{R}^{m\times n}$.
Note that this class of policies is by construction causal, as it enforces the requirement that the control input at time $\tau+k$ is only affected by the reconstructed disturbances up to time $\tau+k-1$.
To obtain a more compact reformulation of the MPC Problem~\eqref{eq:MPC_problem_1}, we can write the system dynamics in matrix form
\begin{equation*}
    \begin{array}{l}
    \boldsymbol{x}_{+}=\mathbf{F} x_{\tau}+\mathbf{G} \boldsymbol{u}+\mathbf{H} \boldsymbol{\eta}\ ,
    \end{array}
\end{equation*}
where the $\vartheta$-dependent matrices $\mathbf{F}$, $\mathbf{G}$ and $\mathbf{H}$ are given by:
\begin{equation*}
\mathbf{F}=
    \begin{bmatrix}
        A \\
        A^{2} \\
        \vdots \\
        A^{T}
    \end{bmatrix}, \ 
 \mathbf{G}=
    \begin{bmatrix}
        B & 0_{n \times m} & \cdots & 0_{n \times m} \\
        A B & B & \ddots & \vdots \\
        \vdots & \ddots & \ddots & 0_{n \times m} \\
        A^{T-1} B & \cdots & A B & B
    \end{bmatrix},
\end{equation*}
\begin{equation*}
\mathbf{H}=
    \begin{bmatrix}
        I_{n \times n} & 0_{n \times n} & \cdots & 0_{n \times n} \\
        A & I_{n \times n} & \ddots & \vdots \\
        \vdots & \ddots & \ddots & 0_{n \times n} \\
        A^{T-1} & \cdots & A & I_{n \times n}
    \end{bmatrix}.
\end{equation*}
Likewise, the input can be parametrized in matrix form as
\begin{equation*}
    \begin{array}{l}
		\boldsymbol{u}=\Gamma+\Lambda \boldsymbol{\eta}\ ,
	\end{array}
\end{equation*}
with
\vspace{-0.1cm}
\begin{equation*}
\Gamma \!=\! \begin{bmatrix}\!
    \gamma_{0} \\
    \gamma_{1} \\
    \vdots \\
    \gamma_{T\!-\!1}
\!\end{bmatrix}\!\!,
\Lambda \!=\!
\begin{bmatrix}\!
    0_{m \times n} & 0_{m \times n} & \cdots & 0_{m \times n} \\
    \lambda_{1,0} & 0_{m \times n} & \ddots & \vdots \\
    \vdots & \ddots & \ddots & 0_{m \times n} \\
    \lambda_{T\!-\!1,0} & \cdots & \lambda_{T\!-\!1, T\!-\!2} & 0_{m \times n}
\!\end{bmatrix}.
\end{equation*}
This results in the FHOC
\begin{equation}\label{eq:MPC_problem_Matrix_form_short}
    \begin{aligned}
        \min_{\Gamma,\Lambda} \quad & 
        \mathbb{E}_{\vartheta,\eta}\left[\tilde{J}\left(\Gamma,\Lambda,\vartheta,\boldsymbol{\eta}\right)\right]\\
        \text{s.t.}\quad & \mathbb{P}_{\vartheta}\left(\mathbb{P}_{\eta|\vartheta}\left(\tilde{f}\left(\Gamma,\Lambda,\vartheta,\boldsymbol{\eta}\right) \leq 0\right) \geq 1-\varepsilon_1\right)\geq 1-\varepsilon_2\\
    \end{aligned}
\end{equation}
where
	\begin{align*}
		\tilde{J}\left(\Gamma,\Lambda,\vartheta,\boldsymbol{\eta}\right) &= J\left(\mathbf{F} x_{\tau}\!+\!\mathbf{G} \Gamma\!+\!\left(\mathbf{H}\!+\!\mathbf{G} \Lambda\right)\!\boldsymbol{\eta},\, \Gamma\!+\!\Lambda \boldsymbol{\eta}\right)\ ,\\
		\tilde{f}\left(\Gamma,\Lambda,\vartheta,\boldsymbol{\eta}\right) &= f\left(\mathbf{F} x_{\tau}\!+\!\mathbf{G} \Gamma\!+\!\left(\mathbf{H}\!+\!\mathbf{G} \Lambda\right)\!\boldsymbol{\eta},\, \Gamma\!+\!\Lambda \boldsymbol{\eta}\right)\ .
	\end{align*}
The solution of the stochastic FHOC Problem~\eqref{eq:MPC_problem_Matrix_form_short} still poses some major challenges. Formulating the chance constraint in a tractable way is only possible under specific assumptions on the disturbance distributions. 
We address these difficulties in the next sections, starting with earlier results on scenario MPC, under the following assumption:
\begin{Assumption}[Convexity]\label{Assumption:Convexity} The cost function $\tilde{J}\left(\Gamma,\Lambda,\vartheta,\boldsymbol{\eta}\right)$ and the constraint function $\tilde{f}\left(\Gamma,\Lambda,\vartheta,\boldsymbol{\eta}\right)$ are convex in the optimization variables $\Gamma$ and $\Lambda$ for almost every $(\vartheta,\boldsymbol{\eta}) \in {\Theta}\times\mathcal{H}^{T}$.\end{Assumption}
Note that, since $(\Gamma,\Lambda)$ parameterize the state and input trajectories $(\boldsymbol{x}_+, \boldsymbol{u})$ through an affine transformation, it is sufficient to require the cost and constraint function to be convex in $(\boldsymbol{x}_+, \boldsymbol{u})$ to satisfy Assumption~\ref{Assumption:Convexity}.
%
%
\section{Background: Scenario MPC with known dynamics}\label{Sec:Background_Scenario}
We consider first the case where the true parameters $\bar{\vartheta}$, i.e. the system matrices $\bar{A}=A(\bar{\vartheta})$ and $\bar{B}=B(\bar{\vartheta})$, are known and discuss how to solve the FHOC Problem~\eqref{eq:MPC_problem_Matrix_form_short} using the scenario approach following~\cite{prandini2012randomized}. 
For a fixed, known, $\vartheta = \bar{\vartheta}$ the FHOC Problem~\eqref{eq:MPC_problem_Matrix_form_short} becomes
\begin{equation}\label{eq:MPC_problem_Matrix_form_known_theta}
	\begin{aligned}
		\min_{\Gamma,\Lambda} \quad & 
		\mathbb{E}_{\eta}\left[\tilde{J}\left(\Gamma,\Lambda,\bar{\vartheta},\boldsymbol{\eta}\right)\right]\\
		\text{s.t.}\quad & \mathbb{P}_{\eta}\left(\tilde{f}\left(\Gamma,\Lambda,\bar{\vartheta},\boldsymbol{\eta}\right) \leq 0\right) \geq 1-\varepsilon_1\\
	\end{aligned}
\end{equation}
The scenario approach~\cite{calafiore2006scenario,campi2008exact, calafiore2010random} is a framework for the solution of general chance-constrained optimization problems. It is a randomized approach based on the solution of a deterministic sampled version of the original optimization problem, called the \textit{Scenario Program} (SP). The scenario approach requires $N$ i.i.d. scenarios, i.e., $N$ samples of the disturbance ${\eta}_k^i\sim \mathbb{P}_{\eta}$, $i=1,\dots,N$, $k=1,\dots,T$.
If the distribution $\mathbb{P}_{\eta}$ is available, scenarios can be obtained by directly sampling from it. Alternatively, if historical data on state-input trajectories is available, scenarios for $\boldsymbol{\eta}$ can be reconstructed from it using~\eqref{eq:dynamics_flip}.
Using the scenarios, we can formulate the following deterministic SP:
\begin{equation}\label{eq:SP_AB_given}
    \begin{aligned}
        \min_{\Gamma,\Lambda} \qquad & \frac{1}{N}\sum_i^{N} \ \tilde{J}\left(\Gamma,\Lambda,\bar{\vartheta},\boldsymbol{\eta}^i \right)\\
        \text{s.t.}\qquad & \tilde{f}\left(\Gamma,\Lambda,\bar{\vartheta},\boldsymbol{\eta}^i\right) \leq 0\quad \text{for}\ i=1,\dots,N.\\
    \end{aligned}
\end{equation}
The expectation in the cost function and the probabilistic constraint have been replaced with their sampled counterparts. The superscript $i$ in $\boldsymbol{\eta}^i$ has been added to highlight the dependency of both functions on the $i^{th}$ scenario realization.

The scenario approach provides probabilistic guarantees on the feasibility of the optimizer of~\eqref{eq:SP_AB_given} under convexity Assumption~\ref{Assumption:Convexity}. Note that we are focusing on providing feasibility guarantees for~\eqref{eq:SP_AB_given}, but it is also possible to obtain guarantees on the worst-case cost by means of an epigraph formulation of the objective function.
We are now in the position to state the following:
\begin{Theorem}\label{Thm:Scenario_MPC_Basic}
\textit{\cite{prandini2012randomized}}
Fix the probability of violation $\varepsilon_1\in[0,1]$, the confidence parameter $\beta\in[0,1]$ and let $d$ be the number of decision variables, i.e., the number of variables in $\Gamma$ and $\Lambda$. Further assume that the SP~\eqref{eq:SP_AB_given} is always feasible and its optimal solution is unique. If $N$ is chosen such that
\begin{equation}\label{eq:N1}
	\sum_{i=0}^{d-1} \begin{pmatrix} N\\ i\end{pmatrix} \varepsilon_1^i (1-\varepsilon_1)^{N-i} \leq \beta\ ,
\end{equation}
then, with probability at least $1-\beta$, the optimal solution to the sampled deterministic SP~\eqref{eq:SP_AB_given} is a feasible solution for the original chance constrained Problem~\eqref{eq:MPC_problem_Matrix_form_short}.
\end{Theorem}
\section{Uncertain Dynamics Scenario-based Stochastic MPC}\label{Sec:Uncertain_dynamics_Scenario_Stochastic_MPC}
In presence of uncertainty on the system dynamics, we aim for a problem formulation that can provide state and input constraint satisfaction with respect to both the epistemic and the stochastic uncertainty. In particular, we want to extend the finite sample result of the previous section to the case where the epistemic uncertainty on the system dynamics is described by a probability distribution on the parameters $\vartheta$, as is common for many probabilistic identification algorithms.
\subsection{Scenario program}
If we assume, for the time being, that we can obtain samples from the joint distribution $\mathbb{P}_{\vartheta,\eta}$, we can write the following deterministic SP:
\begin{equation}\label{eq:SP_eta_theta}
	\begin{aligned}
		\min_{\Gamma,\Lambda} \qquad & \frac{1}{N}\sum_i^{N} \ \tilde{J}\left(\Gamma,\Lambda,{\vartheta}^i,\boldsymbol{\eta}^i \right)\\
		\text{s.t.}\qquad & \tilde{f}\left(\Gamma,\Lambda,{\vartheta}^i,\boldsymbol{\eta}^i\right) \leq 0\quad \text{for}\ i=1,\dots,N.\\
	\end{aligned}
\end{equation}
We can now state the finite sample result that extends Theorem~\ref{Thm:Scenario_MPC_Basic} to include the robustness against the model uncertainty.
\begin{Theorem}\label{Thm:Scenario_MPC_nested}
	Fix the probability of violation $\varepsilon_1\in[0,1]$ and  $\varepsilon_2\in[0,1]$, the confidence parameter $\beta\in[0,1]$ and let $d$ be the number of decision variables. Further assume that the SP~\eqref{eq:SP_eta_theta} is always feasible and its optimal solution is unique. If $N$ is chosen such that
	\begin{equation}\label{eq:N2}
		\sum_{i=0}^{d-1} \begin{pmatrix} N\\ i\end{pmatrix} (\varepsilon_1 \varepsilon_2)^i (1-\varepsilon_1 \varepsilon_2)^{N-i} \leq \beta\ ,
	\end{equation}
	then, with probability at least $1-\beta$, the optimal solution to the sampled deterministic SP~\eqref{eq:SP_eta_theta} is a feasible solution for the chance constrained Problem~\eqref{eq:MPC_problem_Matrix_form_short}.
\end{Theorem}
\begin{proof}
We can apply Theorem~\ref{Thm:Scenario_MPC_Basic} to show that if $N$ satisfies~\eqref{eq:N2}, then with probability $1-\beta$ the optimal solution to the sampled deterministic SP~\eqref{eq:SP_eta_theta} is a feasible solution of the following chance constrained problem:
\begin{equation}\label{eq:MPC_problem_Matrix_form_short_joint}
	\begin{aligned}
		\min_{\Gamma,\Lambda} \qquad & 
		\mathbb{E}_{\vartheta,\eta}\left[\tilde{J}\left(\Gamma,\Lambda,{\vartheta},\boldsymbol{\eta}\right)\right]\\
		\text{s.t.}\qquad & \mathbb{P}_{\vartheta,\eta}\left(\tilde{f}\left(\Gamma,\Lambda,{\vartheta},\boldsymbol{\eta}\right) \leq 0 \right) \geq 1\!-\!\varepsilon_1\varepsilon_2
	\end{aligned}
\end{equation}
We can rewrite the chance constraint in~\eqref{eq:MPC_problem_Matrix_form_short_joint} as
\begin{align}
    & \mathbb{P}_{\vartheta,\eta}\left(\tilde{f}\left(\Gamma,\Lambda,{\vartheta},\boldsymbol{\eta}\right) > 0 \right) \leq \varepsilon_1 \varepsilon_2 \nonumber \\
    \Rightarrow\ & E_{\vartheta,\eta}\left(\mathbb{I}_{\tilde{f} > 0} \right) \leq \varepsilon_1 \varepsilon_2 \nonumber \\
    \Rightarrow\ & E_{\vartheta}\left( E_{\eta | \vartheta} \left( \mathbb{I}_{\tilde{f} > 0} \right)\right) \leq \varepsilon_1 \varepsilon_2 \nonumber \\
    \Rightarrow\ & \frac{1}{\varepsilon_1}E_{\vartheta}\left( E_{\eta | \vartheta} \left( \mathbb{I}_{\tilde{f} > 0} \right)\right) \leq \varepsilon_2 \label{eq:chance_constraint_2}
\end{align}
where $\mathbb{I}_{\tilde{f} > 0}$ is the indicator function of the event $\tilde{f}\left(\Gamma,\Lambda,{\vartheta},\boldsymbol{\eta}\right) > 0$, i.e., 
$$
\mathbb{I}_{\tilde{f} > 0} = \left\{ 
\begin{aligned}
    &1\qquad \tilde{f}\left(\Gamma,\Lambda,{\vartheta},\boldsymbol{\eta}\right) > 0\\
    &0\qquad \mbox{else} 
\end{aligned}
\right.
$$
We can apply Markov's inequality to the left-hand side of~\eqref{eq:chance_constraint_2}
\begin{equation*}
    \mathbb{P}_{\vartheta}\left( E_{\eta | \vartheta} \left( \mathbb{I}_{\tilde{f} > 0} \right)\geq \varepsilon_1 \right) \leq \frac{1}{\varepsilon_1}E_{\vartheta}\left( E_{\eta | \vartheta} \left( \mathbb{I}_{\tilde{f} > 0} \right)\right),
\end{equation*}
leading to
\begin{align*}
    &\mathbb{P}_{\vartheta}\left( E_{\eta | \vartheta} \left( \mathbb{I}_{\tilde{f} > 0} \right)\geq \varepsilon_1 \right)\leq \varepsilon_2\\
    \Rightarrow\ & \mathbb{P}_{\vartheta}\left( E_{\eta | \vartheta} \left( \mathbb{I}_{\tilde{f} > 0} \right) > \varepsilon_1 \right)\leq \varepsilon_2\\
    \Rightarrow\ & \mathbb{P}_{\vartheta}\left( \mathbb{P}_{\eta|\vartheta}\left(\tilde{f}\left(\Gamma,\Lambda,{\vartheta},\boldsymbol{\eta}\right) > 0 \right)> \varepsilon_1 \right)\leq \varepsilon_2\\
    \Rightarrow\ & \mathbb{P}_{\vartheta}\left( \mathbb{P}_{\eta|\vartheta}\left(\tilde{f}\left(\Gamma,\Lambda,{\vartheta},\boldsymbol{\eta}\right) \leq 0 \right)\geq 1-\varepsilon_1 \right)\geq 1-\varepsilon_2
\end{align*}
which shows that a feasible solution to~\eqref{eq:MPC_problem_Matrix_form_short_joint} is also a feasible solution to~\eqref{eq:MPC_problem_Matrix_form_short}.
\end{proof}
Note that, even for the so-called fully supported problems, where the number of supporting constraints equals the number of optimization variables and for which the scenario theory provides a tight bound on the sample size, the bound obtained through Theorem~\ref{Thm:Scenario_MPC_nested} is in general not tight due to the use of Markov's inequality. Furthermore, we can observe that, for a fixed number of scenarios $N$, Theorem~\ref{Thm:Scenario_MPC_Basic} guarantees that, with confidence $1-\beta$, the optimal solution to the SP~\eqref{eq:SP_eta_theta} is a feasible solution of the chance constrained Problem~\eqref{eq:MPC_problem_Matrix_form_short_joint} for any $\bar{\varepsilon}=\varepsilon_1\varepsilon_2$ satisfying~\eqref{eq:N2}.
Thanks to Theorem~\ref{Thm:Scenario_MPC_nested}, we can also see that such solution is feasible for~\eqref{eq:MPC_problem_Matrix_form_short}. This is true for any $\varepsilon_1 \in[0,1]$ and $\varepsilon_2 \in[0,1]$ such that $\varepsilon_1\varepsilon_2\geq\bar{\varepsilon}$. Similarly to the risk-confidence trade-off presented in~\cite{campi2008exact,campi2011sampling}, for fixed number of scenarios $N$ and confidence level $\beta$, it is still possible to adjust the trade-off between the robustness against the epistemic uncertainty on the model and the robustness against the stochastic disturbance.
\subsection{Scenarios construction}
We now turn to the question of how to obtain the samples from the joint distribution $\mathbb{P}_{\vartheta,\eta}$.

If the distributions $\mathbb{P}_{\vartheta}$ and $\mathbb{P}_{\eta}$ are available, we can directly generate samples from the joint distribution. If not, we can generate samples from historical data.

When identification or learning algorithms are employed to obtain a probabilistic description of the plant dynamics, historical data are used to produce distributions of models or directly samples from such distributions. 
An example of a procedure to generate samples of the dynamics based on bootstrapping is given in Section~\ref{Sec:Numerics}.

Having obtained samples of the system dynamics, one can generate the corresponding samples $\boldsymbol{\eta}$ using more historical data. To maintain the independency requirement necessary for Theorem~\ref{Thm:Scenario_MPC_nested}, we require access to $N$ i.i.d. historical state-input trajectory samples, each of length $T$ (the prediction horizon). We denote these by  $\{\boldsymbol{x}^i,\boldsymbol{u}^i\}_{i=1}^N$ with $\boldsymbol{x}^i=\{x^i_k\}_{k=0}^T$, $\boldsymbol{u}^i=\{u^i_k\}_{k=0}^{T-1}$. With these ingredients one can generate the required scenarios using the following algorithm.
\begin{Algorithm}[Sequential sampling for Uncertain Dynamics Scenario-based MPC]{}\label{Algo:1}
	For $i=1,\dots,N$
	\begin{enumerate}
		\item Sample $\vartheta^i\sim \mathbb{P}_{\vartheta}$ from the parameters distribution;
		\item Obtain noise realizations of length $T$ from historical data conditioned on the sampled parameters ${\vartheta}^i$ as\\
		$\eta_k^i = x^i_{k+1}-\left(A^{\vartheta^i}x^i_k + B^{\vartheta^i}u^i_{k}\right)$, $k=0,\dots,T$.\\
		These are samples of $\boldsymbol{\eta}^i\sim \mathbb{P}_{\eta|\vartheta = \vartheta^i}$.
	\end{enumerate}
\end{Algorithm}
\subsection{Computational complexity}
Following~\cite{campi2008exact}, the number of scenarios $N$ required to satisfy~\eqref{eq:N2} can be explicitly upper bounded by
\begin{equation*}
\tilde{N} = \frac{2}{\varepsilon_1 \varepsilon_2} \left(d+\ln \frac{1}{\beta}\right)\ .
\end{equation*}
We can observe that $\tilde{N}$ grows logarithmically with  $\frac{1}{\beta}$ and linearly in $d$ and in $\frac{1}{\varepsilon_1 \varepsilon_2}$. For the disturbance-affine control policy~\eqref{eq:AffinePolicy} presented in Section~\ref{Sec:Problem_Formulation}, the number of optimization variables $d=mT + mn\frac{(T-1)T}{2}$ grows linearly with the state and input dimensions and quadratically with the prediction horizon length $T$. Even though the asymptotic growth is modest, practically applying Theorem~\ref{Thm:Scenario_MPC_nested} would require solving a large program and access to large number of historical system trajectories to generate the samples of $\boldsymbol{\eta}$ using Algorithm~\ref{Algo:1}. The number of required scenarios can be decreased by allowing larger probabilities of violations or by reducing the number of optimization variable, e.g. limiting $\Lambda$ to have non zero elements only on the sub-diagonal to achieve linear growth in the prediction horizon length $T$.
Nonetheless, empirically good performance can be obtained even for small number of scenarios, well below the minimum number required by the theoretical result of Theorem~\ref{Thm:Scenario_MPC_nested}, as seen in the numerical example below. 
\section{Numerical example}\label{Sec:Numerics}
We consider the system
\begin{equation}\label{eq:Dynamics_example}
	x_{k+1}=
	\begin{bmatrix} 
		0.9 & 0.15\\
		0.05 & 0.9\end{bmatrix} x_{k} + 
	\begin{bmatrix}
		0 \\
		1 \\
	\end{bmatrix} u_{k} +\eta_k,
\end{equation}
with additive disturbance $\eta_k \sim \mathcal{U}[-0.02;+0.02]$.
We assume that we have access to an identification dataset $\mathcal{D}_{Id}$ of size $L_{Id}=50$ comprising $\{ \boldsymbol{x},\boldsymbol{u},\boldsymbol{x}_+  \}^j$, $j=1,\dots,L_{Id},$ collected by applying a random input sequence \mbox{$u_k \sim \mathcal{U}[-0.5;+0.5]$}, $k=1,\dots,L_{Id}$, to the \textit{true} dynamical system starting from a random initial condition $x_0\sim\mathcal{U}[-0.5;+0.5]$. The small size of this dataset can generally represent a poorly designed identification experiment, with either a weakly exciting signal, correlated measurements or low signal-to-noise ratio. Using this data we generate $N$ samples of $\vartheta^i$ using the following bootstrapping algorithm.
\begin{Algorithm}[Empirical Bootstrap]\label{Algo:Bootstrap}$ $
	\begin{enumerate}
		\item Generate $N$ bootstrapped datasets $\mathcal{D}_{BS}^i, \ i=1,\dots,N$ of size $L_{Id}$ by sampling with replacement from $\mathcal{D}_{Id}$.
		\item Use the Least Squares method to obtain one point estimate of the parameter $\hat{\vartheta}_{BS}^i$ for each $\mathcal{D}_{BS}^i$, \mbox{$i=1,\dots,N$}.
	\end{enumerate}
\end{Algorithm}
We also assume availability of a dataset $\mathcal{D}_{Hist}$ comprising $N$ historical trajectory data of length $T$ to obtain samples of the disturbance using Algorithm~\ref{Algo:1}.

We will compare three methods:
\begin{enumerate}
	\item Uncertain Dynamics Scenario MPC (\textbf{UD-SMPC}): it relies on the solution of the deterministic SP~\eqref{eq:SP_eta_theta}, using the bootstrap Algorithm~\ref{Algo:Bootstrap} to generate samples of the uncertain dynamics from the given dataset $\mathcal{D}_{Id}$;
	\item Least Squares Scenario MPC (\textbf{LS-SMPC}): it relies on the solution of the deterministic SP~\eqref{eq:SP_AB_given}, considering as the deterministic model the one obtained from the identification dataset $\mathcal{D}_{Id}$ by Least Squares regression;
	\item Ground Truth Scenario MPC (\textbf{GT-SMPC})~\cite{prandini2012randomized}: it relies on the solution of the deterministic SP~\eqref{eq:SP_AB_given}, using the ground truth model~\eqref{eq:Dynamics}. This is unrealistic in practice and it is provided only for the sake of comparison as an ideal baseline. 
\end{enumerate}
For all the following simulations we consider initial condition $x_0 =[x_0^{(1)}, x_0^{(2)}]^\top = [ 0.6, 0 ]^\top$, MPC horizon $T = 5$ and cost function $J=\boldsymbol{x}_+^\top Q \boldsymbol{x}_+ + \boldsymbol{u}^\top R \boldsymbol{u}$ 
with diagonal weighting matrices $R=0.1\, {I}_{T\times T}$ and $Q = {I}_{2T\times 2T}$. To obtain a simple comparison we will also consider only the constant constraint $f\left( \boldsymbol{x}_+, \boldsymbol{u}\right) = \max([0.5 - \boldsymbol{x}_+^{(1)}, - \boldsymbol{x}_+^{(2)}]^\top) \leq 0 $, i.e. we require the first and the second elements of the state to be greater than $0.5$ and $0$ respectively.
Theorem~\ref{Thm:Scenario_MPC_nested} requires SP~\eqref{eq:SP_eta_theta} to be always feasible; we deal with this issue by introducing a slack formulation for the chance constraint satisfaction $f\left( \boldsymbol{x}_+, \boldsymbol{u}\right) \leq \sigma$, with $\sigma$ added as a linear penalty in the cost function with a weight of $10^5$, to avoid violating the constraint whenever possible.
\subsection{Open-loop analysis}
To show how the inclusion of model uncertainty improves the reliability of the open-loop prediction we compare the three methods by computing the empirical violation probability over $1000$ realizations of the \textit{true} system \mbox{$T$-steps} trajectory.
We repeat this for $400$ Monte Carlo (MC) realizations of $\mathcal{D}_{Id}$ and $\mathcal{D}_{Hist}$. The number of scenarios is determined by setting $\beta = 10^{-5}$ and $d=26$, $25$ for the disturbance-affine policy~\eqref{eq:AffinePolicy} and one for the slack variable $\sigma$. Following~\eqref{eq:N2}, the number of scenarios needed to attain $\varepsilon_1=0.1$ and $\varepsilon_2=0.3$ for \textbf{UD-SMPC} is $1776$, whereas, following~\eqref{eq:N1} the number of scenarios needed to attain $\varepsilon_1=0.1$, for the other two methods is $523$.
Figure~\ref{fig:OpenLoopProbOfViolation} shows the cumulative distribution of the empirical probability of violation. For approximately $12\%$ of the datasets MC realizations, the \textbf{LS-SMPC} optimal policy results in a probability of violation greater than the specified $\varepsilon_1=0.1$. This issue is related to the neglected model uncertainty and cannot be dealt with by simply increasing the number of scenarios. \textbf{UD-SMPC} robustifies against the epistemic uncertainty by sampling the uncertain dynamics, resulting in lower violation probabilities.
\vspace{-0.1cm}
\begin{figure}[h]
\includegraphics[width=0.96\columnwidth]{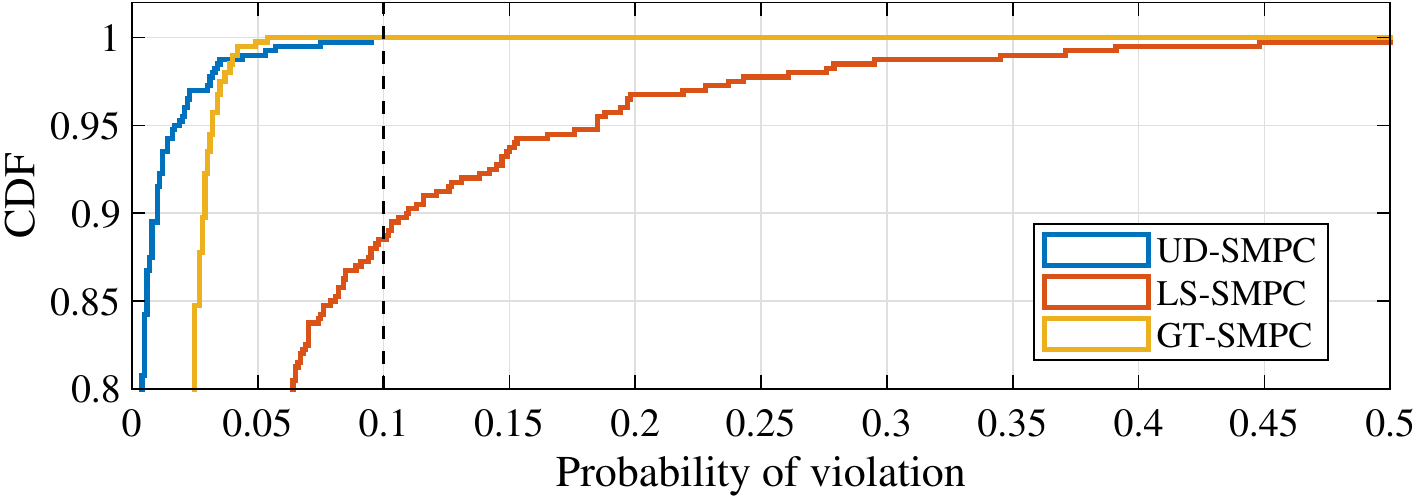}
\vspace{-0.1cm}
\caption{Cumulative empirical violation probability of the optimal feedback policies.}\label{fig:OpenLoopProbOfViolation}
\end{figure}
\vspace{-0.1cm}
%
%
\subsection{Closed-loop analysis}
We compare the MPC $10$-steps closed-loop cost and number of violations of the three methods for equal number of scenarios $N=\{ 2^6, 2^7, 2^8, 2^9, 2^{10}, 2^{11} \}$. We repeat this for $400$ MC realizations of $\mathcal{D}_{Id}$ and $\mathcal{D}_{Hist}$. The analysis of the closed-loop performance is complicated by confounding factors such as:
\begin{enumerate}
	\item the receding horizon nature of MPC, that inherently provides a certain degree of robustness;
	\item the fact that problem we are trying to solve is not necessarily fully supported, hence our estimates of $N$ are not tight;
	\item the additional conservatism due to the use of Markov's inequality in the proof of Theorem~\ref{Thm:Scenario_MPC_nested}.
\end{enumerate}
In Figure~\ref{fig:ClosedLoopViolation} we can observe that the closed-loop violations decrease with larger number of scenarios. This in turns induces an increase in the closed-loop cost as shown in Figure~\ref{fig:ClosedLoopCostComparison1} and Figure~\ref{fig:ClosedLoopCostComparison2}. 
By comparing \textbf{UD-SMPC} and \textbf{LS-SMPC} performance, we can see that explicitly considering the uncertainty in the dynamics with \textbf{UD-SMPC} allows us to obtain a more robust controller, leading to fewer closed-loop violations and less outliers in the cost function. 
\begin{figure}[h]
\includegraphics[width=0.96\columnwidth]{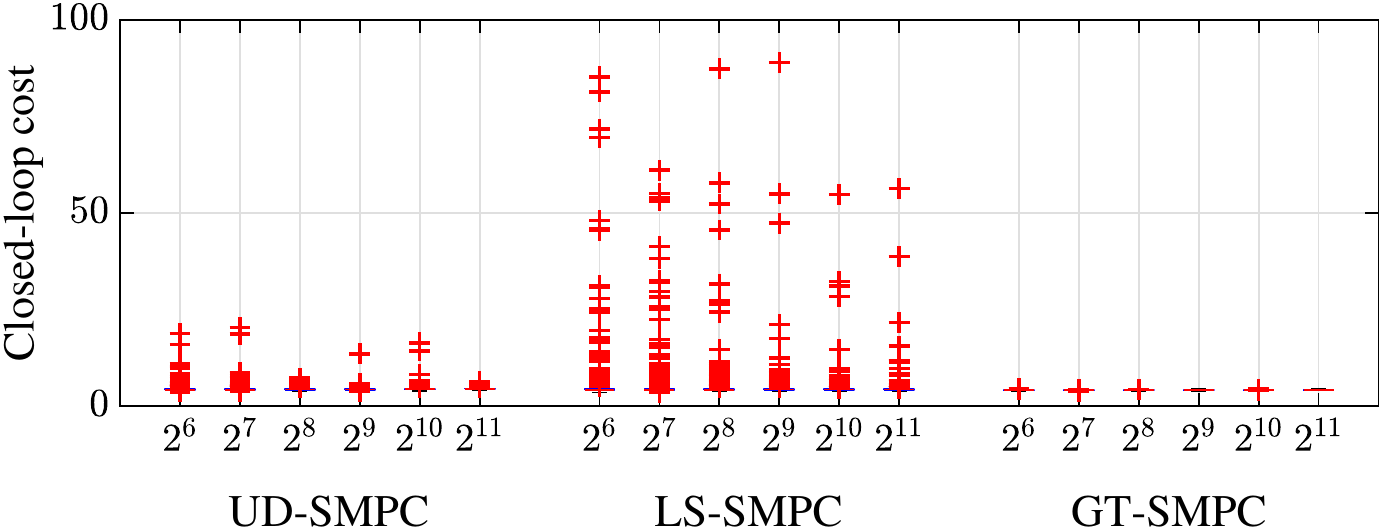}
\vspace{-0.1cm}
\caption{MPC $10$-steps closed-loop cost comparison for number of scenarios $N=\{ 2^6, 2^7, 2^8, 2^9, 2^{10}, 2^{11} \}$.}\label{fig:ClosedLoopCostComparison1}
\end{figure}
\vspace{-0.2cm}
\begin{figure}[h]
\includegraphics[width=0.96\columnwidth]{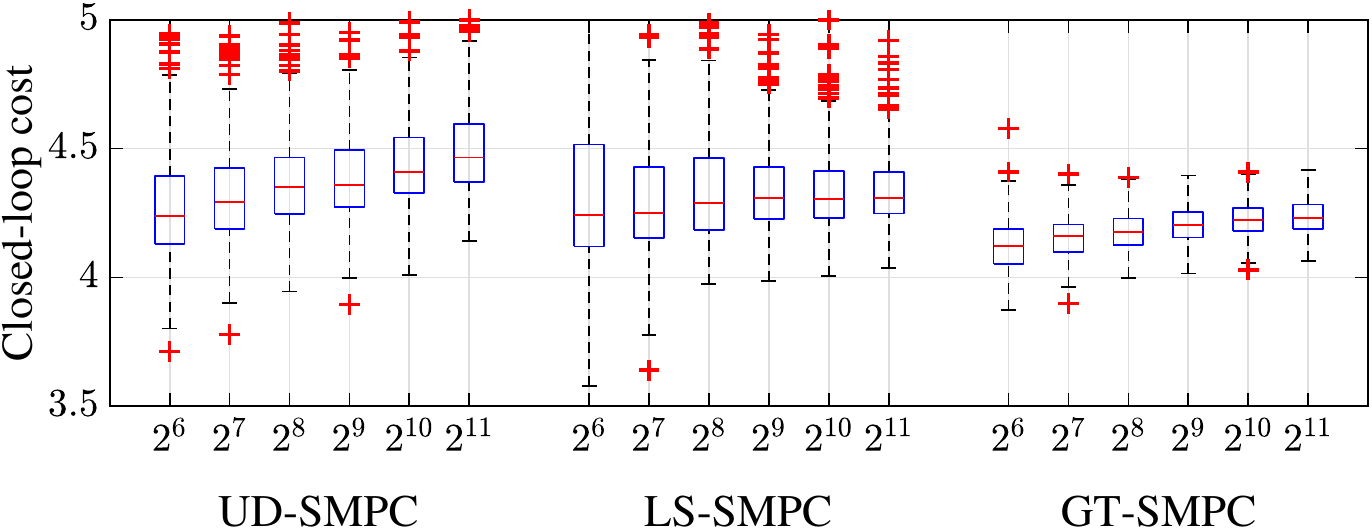}
\vspace{-0.1cm}
\caption{Close-up: Comparison of the MPC $10$-steps closed-loop cost.}\label{fig:ClosedLoopCostComparison2}
\end{figure}
\vspace{-0.2cm}
\begin{figure}[h]
	\includegraphics[width=0.96\columnwidth]{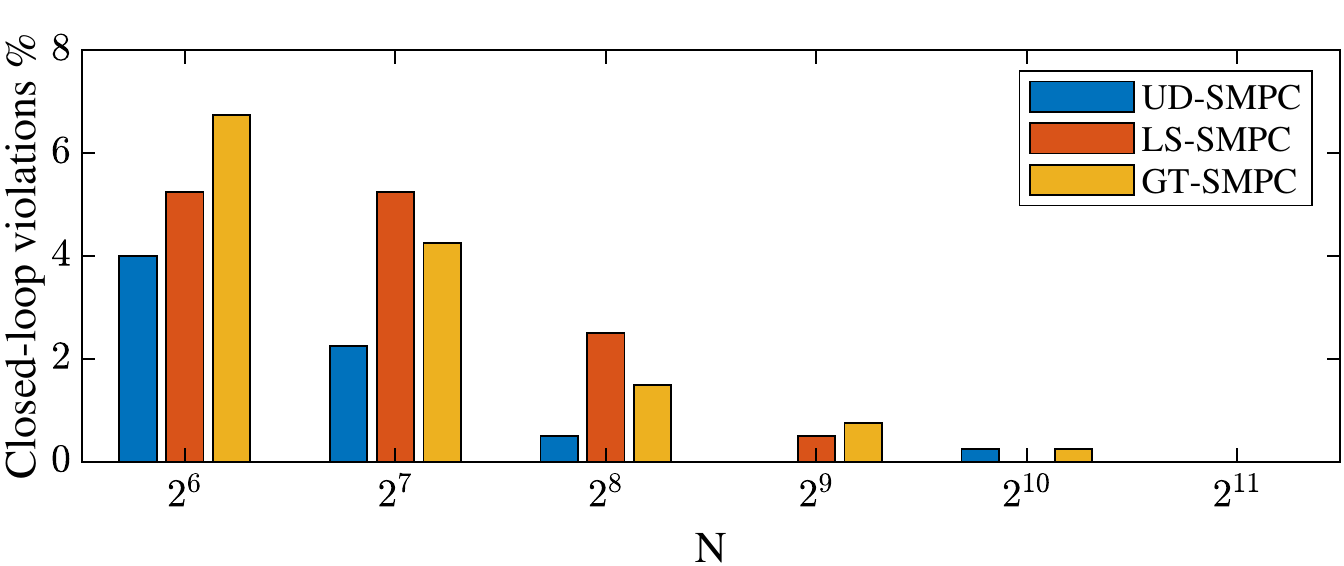}
	\vspace{-0.1cm}
	\caption{Comparison of the MPC $10$-steps closed-loop number of violations for $N=\{ 2^6, 2^7, 2^8, 2^9, 2^{10}, 2^{11} \}$.}\label{fig:ClosedLoopViolation}
\end{figure}
\vspace{-0.1cm}
%
\section{Conclusions and future work}\label{Sec:Conclusion}
We developed a general scenario-based optimization framework for the solution of chance constrained stochastic MPC with uncertain dynamics. We extended previous work in scenario-based stochastic MPC providing a principled way of dealing with the epistemic uncertainty related to the system dynamics. Unlike stochastic and robust MPC approaches that typically require additional assumptions or over-approximations of the chance constraints, the proposed method allows for arbitrary distributions of the stochastic parameters. Moreover, the explicit distributions of the disturbances and of the model parametric uncertainties are not required, as long as samples are available.

Interesting future work directions include the extension of the proposed approach to an output feedback setup and the adaptation of the proposed scheme to an online learning dual formulation setting, in which at each step there is a trade-off between exploration and exploitation. 
One of the main aspects that remains to be addressed is the recursive (probabilistic) feasibility through the definition of safe probabilistic terminal sets. This could improve the applicability of the proposed method, making it suitable for a wide range of safety-critical learning-based applications.
\bibliographystyle{IEEEtran}
\bibliography{bib}

\begin{thebibliography}{10}
\providecommand{\url}[1]{#1}
\csname url@samestyle\endcsname
\providecommand{\newblock}{\relax}
\providecommand{\bibinfo}[2]{#2}
\providecommand{\BIBentrySTDinterwordspacing}{\spaceskip=0pt\relax}
\providecommand{\BIBentryALTinterwordstretchfactor}{4}
\providecommand{\BIBentryALTinterwordspacing}{\spaceskip=\fontdimen2\font plus
\BIBentryALTinterwordstretchfactor\fontdimen3\font minus
  \fontdimen4\font\relax}
\providecommand{\BIBforeignlanguage}[2]{{%
\expandafter\ifx\csname l@#1\endcsname\relax
\typeout{** WARNING: IEEEtran.bst: No hyphenation pattern has been}%
\typeout{** loaded for the language `#1'. Using the pattern for}%
\typeout{** the default language instead.}%
\else
\language=\csname l@#1\endcsname
\fi
#2}}
\providecommand{\BIBdecl}{\relax}
\BIBdecl

\bibitem{wan2003efficient}
Z.~Wan and M.~V. Kothare, ``An efficient off-line formulation of robust model
  predictive control using linear matrix inequalities,'' \emph{Automatica},
  vol.~39, no.~5, pp. 837--846, 2003.

\bibitem{kouvaritakis2000efficient}
B.~Kouvaritakis, J.~A. Rossiter, and J.~Schuurmans, ``Efficient robust
  predictive control,'' \emph{IEEE Transactions on automatic control}, vol.~45,
  no.~8, pp. 1545--1549, 2000.

\bibitem{mayne2000constrained}
D.~Q. Mayne, J.~B. Rawlings, C.~V. Rao, and P.~O. Scokaert, ``Constrained model
  predictive control: Stability and optimality,'' \emph{Automatica}, vol.~36,
  no.~6, pp. 789--814, 2000.

\bibitem{cannon2010stochastic}
M.~Cannon, B.~Kouvaritakis, S.~V. Rakovi{\'c}, and Q.~Cheng, ``Stochastic tubes
  in model predictive control with probabilistic constraints,'' \emph{IEEE
  Transactions on Automatic Control}, vol.~56, no.~1, pp. 194--200, 2010.

\bibitem{cinquemani2009convex}
E.~Cinquemani, M.~Agarwal, D.~Chatterjee, and J.~Lygeros, ``On convex problems
  in chance-constrained stochastic model predictive control,'' \emph{arXiv
  preprint arXiv:0905.3447}, 2009.

\bibitem{alamo2010sample}
T.~Alamo, R.~Tempo, and A.~Luque, ``On the sample complexity of randomized
  approaches to the analysis and design under uncertainty,'' in
  \emph{Proceedings of the 2010 American Control Conference}.\hskip 1em plus
  0.5em minus 0.4em\relax IEEE, 2010, pp. 4671--4676.

\bibitem{vidyasagar2001randomized}
M.~Vidyasagar, ``Randomized algorithms for robust controller synthesis using
  statistical learning theory,'' \emph{Automatica}, vol.~37, no.~10, pp.
  1515--1528, 2001.

\bibitem{calafiore2006probabilistic}
G.~Calafiore and F.~Dabbene, \emph{Probabilistic and randomized methods for
  design under uncertainty}.\hskip 1em plus 0.5em minus 0.4em\relax Springer,
  2006.

\bibitem{shapiro2005complexity}
A.~Shapiro and A.~Nemirovski, ``On complexity of stochastic programming
  problems,'' in \emph{Continuous optimization}.\hskip 1em plus 0.5em minus
  0.4em\relax Springer, 2005, pp. 111--146.

\bibitem{calafiore2011research}
G.~C. Calafiore, F.~Dabbene, and R.~Tempo, ``Research on probabilistic methods
  for control system design,'' \emph{Automatica}, vol.~47, no.~7, pp.
  1279--1293, 2011.

\bibitem{calafiore2006scenario}
G.~C. Calafiore and M.~C. Campi, ``The scenario approach to robust control
  design,'' \emph{IEEE Transactions on automatic control}, vol.~51, no.~5, pp.
  742--753, 2006.

\bibitem{campi2008exact}
M.~C. Campi and S.~Garatti, ``The exact feasibility of randomized solutions of
  uncertain convex programs,'' \emph{SIAM Journal on Optimization}, vol.~19,
  no.~3, pp. 1211--1230, 2008.

\bibitem{calafiore2010random}
G.~C. Calafiore, ``Random convex programs,'' \emph{SIAM Journal on
  Optimization}, vol.~20, no.~6, pp. 3427--3464, 2010.

\bibitem{prandini2012randomized}
M.~Prandini, S.~Garatti, and J.~Lygeros, ``A randomized approach to stochastic
  model predictive control,'' in \emph{2012 IEEE 51st IEEE Conference on
  Decision and Control (CDC)}.\hskip 1em plus 0.5em minus 0.4em\relax IEEE,
  2012, pp. 7315--7320.

\bibitem{calafiore2012robust}
G.~C. Calafiore and L.~Fagiano, ``Robust model predictive control via scenario
  optimization,'' \emph{IEEE Transactions on Automatic Control}, vol.~58,
  no.~1, pp. 219--224, 2012.

\bibitem{calafiore2013stochastic}
------, ``Stochastic model predictive control of lpv systems via scenario
  optimization,'' \emph{Automatica}, vol.~49, no.~6, pp. 1861--1866, 2013.

\bibitem{margellos2014road}
K.~Margellos, P.~Goulart, and J.~Lygeros, ``On the road between robust
  optimization and the scenario approach for chance constrained optimization
  problems,'' \emph{IEEE Transactions on Automatic Control}, vol.~59, no.~8,
  pp. 2258--2263, 2014.

\bibitem{goulart2006optimization}
P.~J. Goulart, E.~C. Kerrigan, and J.~M. Maciejowski, ``Optimization over state
  feedback policies for robust control with constraints,'' \emph{Automatica},
  vol.~42, no.~4, pp. 523--533, 2006.

\bibitem{campi2011sampling}
M.~C. Campi and S.~Garatti, ``A sampling-and-discarding approach to
  chance-constrained optimization: feasibility and optimality,'' \emph{Journal
  of optimization theory and applications}, vol. 148, no.~2, pp. 257--280,
  2011.

\end{thebibliography}
\end{document}